\documentclass[12pt,twoside,leqno]{amsart}  %leqno = Die Formelnummern links 
\usepackage{amsmath,amscd,amssymb,amsfonts}
\newcommand{\sN}{{\mathcal N}}

\newcommand{\sX}{{\mathcal X}}

\newcommand{\C}{{\mathbb C}}

\newcommand{\N}{{\mathbb N}}
\renewcommand{\P}{{\mathbb P}}

\newcommand{\R}{{\mathbb R}}

\newcommand{\Z}{{\mathbb Z}}
\newcommand{\lra}{\longrightarrow }

\theoremstyle{plain}
\newtheorem{thm}{Theorem}
\newtheorem{lemma}[thm]{Lemma}
\newtheorem{notation}[thm]{Notation}

\newtheorem{condition}{Condition}
\newtheorem{prop}[thm]{Proposition}

\newtheorem{definition}[thm]{Definition}

\numberwithin{thm}{section}
\numberwithin{equation}{section}

\begin{document}
\title{Dwork Congruences and Reflexive Polytopes}
\author{Kira Samol}
\author{Duco van Straten}

\thanks{This work is entirely supported by DFG Sonderforschungsbereich/Transregio 45}                           %= Fussnotentext.
\begin{abstract}
We show that the coefficients of the power series expansion of the principal
period of a Laurent polynomial satisfy strong congruence properties.
These congruences play key role in the explicit $p$-adic analytic continuation 
of the unit-root. The methods we use are completely elementary.     
\end{abstract}

\maketitle

\section{Dwork congruences}
\begin{definition}
Let $(a(n))_{n \in \N_0}$ be a sequence of integers with $a(0)=1$ and let
$p$ be a prime number.
We say that $(a(n))_n$ satisfies the {\em Dwork congruences} if  for all
$s,m,n \in \N_0$ one has
\begin{enumerate}
\item[D1]  \[\frac{a(n)}{a(\lfloor n/p \rfloor)} \in \Z_p\]
%\item[D2] \[a(n+mp^s)a(\lfloor n/p\rfloor)\equiv a(\lfloor n/p\rfloor +mp^{s-1})a(n)\mod p^s\]
\item[D2] \[\frac{a(n+mp^{s+1})}{a(\lfloor n/p \rfloor +mp^s)}\equiv\frac{a(n)}{a(\lfloor n/p \rfloor)} \mod p^{s+1}\]
\end{enumerate}
\end{definition}
In fact, the validity of these congruences is implied by those for which $n<p^{s+1}$, as one sees by writing $n=n'+mp^{s+1}$ with $n' <p^{s+1}$.
By cross-multiplication, $D2$ becomes
\begin{enumerate}
\item[D3] \[a(n+mp^{s+1})a(\lfloor \frac{n}{p} \rfloor) \equiv a(n)a(\lfloor \frac{n}{p}\rfloor+mp^s) \mod p^{s+1}.\]
\end{enumerate}
%and in the presence of $D1$ the congruences $D2$ and $D3$ are equivalent.\\

The congruences for $s=0$ say that for $0 \le n_0 \le p-1$ one has
 \[a(n_0+mp) \equiv a(n_0)a(m) \mod p\]
So if we write $n$ in base $p$
 \[n=n_0+pn_1+\ldots+n_r p^r,\;\;\;0 \le n_i \le p-1,\] 
 we find by repeated application that
\[a(n) \equiv a(n_0)a(n_1)\ldots a(n_r) \mod p\]
In fact, this is easily seen to be equivalent to $D3$ for $s=0$.

Similarly, for higher $s$ the congruences $D3$ are equivalent to
\[a(n_0+...+n_{s+1}p^{s+1})a(n_1+...+n_{s}p^{s-1})\equiv\]\begin{equation} a(n_0+...+n_{s}p^{s})a(n_1+...+n_{s+1}p^{s}) \mod p^{s+1}.\label{cong1}\end{equation}

%In fact, given the integrality condition $D1$, these congruences are 
%equivalent to $D2$. 
The congruences express a strong $p$-adic analyticity property of the 
function 
\[n \mapsto a(n)/a(\lfloor n/p \rfloor)\]
and play a key role in the $p$-adic analytic continuation of the series
\[F(t) =\sum_{n=0}^{\infty} a(n)t^n\]
to points on the closed $p$-adic unit disc.
More precisely, one has the following theorem (see \cite{dwork}, Theorem 3.)

\begin{thm}\label{dworksatz3}
Let $(a(n))_n$ be a $\mathbb{Z}_p-$valued sequence satisfying the Dwork congruences {\em D1} and {\em D2 }. Let
\[F(t)=\sum_{n=0}^{\infty}a(n)t^n\]
and
\[F^s(t)=\sum_{n=0}^{p^s-1}a(n)t^n.\]
Let ${\frak D}$ be the region in $\mathbb{Z}_p$
\[{\frak D}:=\{x \in \mathbb{Z}_p, |F^1(x)|=1\}.\]
Then $F(t)/F(t^p)$ is the restriction to $p\mathbb{Z}_p$ of an analytic element $f$ of support ${\frak D}$:
\[f(x)=\lim_{s\rightarrow\infty}F^{s+1}(x)/F^s(x^p).\]
\end{thm}

The congruences were used in \cite{SvS} to determine Frobenius polynomials
associated to Calabi-Yau motives coming from fourth order operators of
Calabi-Yau type from the list \cite{AESZ}. 
Although there are many examples of sequences that satisfy these congruences, 
the true cohomological meaning remains obscure at present. For a recent
interpretation in terms of formal groups, see \cite{yu}.
In this paper we will give a completely elementary proof of the congruences
$D3$ for sequences $(a(n))_n$ that arise as constant term of the powers of
a fixed Laurent polynomial with integral coefficients and whose Newton
polyhedron contains a unique interior point. These include the series that
come from reflexive polytopes. 
  
\section{Laurent polynomials}
We will use the familiar multi-index notation 
for monomials and exponents
\[ X^{\bf a}=X_1^{a_1}X_2^{a_2}\ldots X_n^{a_n},\;\;\;{\bf a}=(a_1,a_2,\ldots,a_n) \in \Z^n\] 
to write a general Laurent-polynomial as
\[f =\sum_{\bf a} c_{\bf a}X^{\bf a} \in \Z[X_1,X_1^{-1},X_2,X_2^{-1},\ldots,X_n,X_n^{-1}].\] 

The {\em support} of $f$ is the set of exponents $\bf a$ occuring in $f$, i.e. 
\[ \textup{supp}(f):=\{{\bf a} \in \Z^n\;|\;c_{\bf a} \neq 0\}\]
The {\em Newton polyhedron} $\Delta(f) \subset \R^n$ of $f$ is defined as the
convex hull of its support
\[ \Delta(f):=\textup{convex}(\textup{supp}(f))\]

When the support of $f$ consists of $m$ monomials, we can put the
information of the polyhedron $\Delta:=\Delta(f)$ in an $n \times m$ matrix $\mathcal{A} \in Mat(m \times n, \Z)$, whose columns ${\bf a}_j$, 
$j=1,2,\ldots,m$ are the exponents of $f$;
\[\mathcal{A}=\left({\bf a}_1,{\bf a}_2,\ldots,{\bf a}_m \right)=\left(
\begin{array}{cccc}
a_{1,1}&a_{1,2}&\ldots&a_{1,m}\\
\vdots&        &       & \vdots\\
a_{n,1}&a_{1,2}&\ldots&a_{n,m}\\
\end{array} \right)\]
so that we can write 
\[f=\sum_{j=1}^m c_j X^{{\bf a}_j}=\sum_{j=1}^m c_j \prod_{i=1}^n X^{a_{i,j}}\]
The polyhedron $\Delta$ is the image of the standard simplex $\Delta_m$
under the map
\[\R^m \stackrel{\mathcal{A}}{\lra} \R^n\]

The following theorem will play a key role in the sequel.

\begin{thm}\label{reflexivsatz}
Let $\Delta$ be an integral polyhedron with $0$ as unique interior
point. Then for all non-negative integral vectors 
$(\ell_1,\ell_2,\ldots,\ell_m) \in \Z^m$ such that $\sum_{i=1}^ma_{i,j}\ell_j\not=0$ for some $1\leq i \leq n$, one has
\[
\gcd_{i=1,...,n}(\sum_{j=1}^m a_{i,j}\ell_j)\leq \sum_{j=1}^m\ell_j\]

\end{thm}

\begin{proof}
Assume that there exists a non-negative integral vector $\ell=(\ell_1,...,\ell_m) \in \Z^m$ such that $\sum_{i=1}^ma_{i,j}\ell_j\not=0$ for some $1\leq i \leq n$ and
\[g:=\gcd_{i=1,...,n}(\sum_{j=1}^ma_{i,j}\ell_j)>\sum_{j=1}^m\ell_j.\] 
We have
\[{\bf a}_1\ell_1+...+{\bf a}_m\ell_m=
\mathcal{A} \left(\begin{array}{c}\ell_1 \\ \vdots \\ \ell_m\end{array}\right)=\left(\begin{array}{c}\sum_{j=1}^ma_{1,j}\ell_j \\ \vdots \\ \sum_{j=1}^ma_{n,j}\ell_j\end{array}\right).\]
The components of the vector at the right hand side are all divisible by $g$, so that after division
by $g$ we obtain a non-zero lattice point
\[v:=\frac{\ell_1}{g}{\bf a_1}+...+\frac{\ell_m}{g}{\bf a}_m \in \Z^n \]
of $\Delta$ with
\[\sum_{j} \frac{\ell_j}{g} < 1\]
The interior points of $\Delta$ (i.e. the points that do not lie on the boundary) 
consist of the combinations 
\[\alpha_1{\bf a}_1+...+\alpha_m{\bf a}_m\] 
of the columns of $\mathcal{A}$ with $\sum_{j=1}^m\alpha_j<1$.
As $0$ was assumed to be the only interior lattice point of $\Delta$ we
arrive at a contradiction. 
\end{proof}

We remark that the above statement applies in particular to {\em reflexive polyhedra}.
%\cite{batyrev}.

\section{The fundamental period}

\begin{notation}
For a Laurent-polynomial we denote by $[f]_0$ the {\em constant term}, that is, the
coefficient of the monomial $X^0$. 
\end{notation}

\begin{definition}
The {\em fundamental period} of $f$ is the series
\[ \Phi(t):=\sum_{k=0}^{\infty} a(k) t^k, \;\;\;a(k):=[f^k]_0\]
\end{definition}

Note that the function $\Phi(t)$ can be interpreted as the period of a holomorphic
differential  form on the hypersurface
$X_t:=\{t.f=1\} \subset (\C^*)^n$, as one has
\[
\begin{array}{rcl}
\Phi(t)&=&\sum_{k=0}^{\infty} [f^k]_0 t^k\\[2mm]
       &=&\sum_{k=0}^{\infty}\frac{1}{(2\pi i)^n}\int_T f^k t^k \Omega\\[2mm]
       &=&\frac{1}{(2\pi i)^n} \int_T\sum_{k=0}^{\infty} f^k t^k \Omega\\[2mm]
       &=&\frac{1}{(2\pi i)^n} \int_T \frac{1}{1-t f}\Omega\\[2mm]
       &=&\int_{\gamma_t} \omega_t\\
\end{array}
\]
Here $\Omega:=\frac{dX_1}{X_1}\frac{dX_2}{X_2}\ldots\frac{dX_n}{X_n}$, $T$ is the cycle
 given by $|X_i| =\epsilon_i$ and homologous to the Leray coboundary
of $\gamma_t \in H_{n-1}(X_t)$ and
\[\omega_t=Res_{X_t} (\frac{1}{1-t f}\Omega)\]

In particular, $\Phi(t)$ is a solution of a Picard-Fuchs equation; the coefficients
$a(k)$ satisfy a linear recursion relation.

\begin{thm}\label{hauptsatz} Let $f \in \Z[X_1,X_1^{-1},\ldots, X_n, X_n^{-1}]$ with integral coefficients.
Assume that the Newton polyhedron $\Delta(f)$ has $0$ as its unique interior lattice point.\\
Then the coefficients $a(n)=[f^n]_0$ of the fundamental period satisfy for each prime number
$p$ and $s \in \N$ the congruence
\[a(n_0+...+n_sp^s)a(n_1+...+n_{s-1}p^{s-2})\equiv\]\begin{equation} a(n_0+...+n_{s-1}p^{s-1})a(n_1+...+n_sp^{s-1}) \mod p^s.\label{cong1}\end{equation}
where $0 \le n_i \le p-1$ for $0\leq i \leq s-1$.
\end{thm}

We remark that already for the simplest cases where the the Newton polyhedron contains
more than one lattice point, like $f=X^2+X^{-1}$, the coefficients $a(n)$ do not
satisfy such simple congruences.
 
\section{Proof for the congruence $\mod p$ }
For $s=1$ we have to show that for all $n_0 \le p-1$ 
\[a(n_0+n_1 p) \equiv a(n_0) a(n_1) \mod p,\]

The proof we will give is completely elementary; the key ingredient is theorem 
\ref{reflexivsatz}, which states that
for all non-negative integral $\ell=(\ell_1,...,\ell_m)$ one has,
\[\gcd_{i=1,...,n}(\sum_{j=1}^ma_{i,j}\ell_j)\leq \sum_{j=1}^m\ell_j\]

\begin{prop} Let $f$ be a Laurent polynomial as above and $n_0 < p$. Then
\[\left[f^{n_0}f^{n_1p}\right]_0\equiv \left[f^{n_0}\right]_0\left[f^{n_1}\right]_0\mod p.\]
\end{prop} 

\begin{proof}
As $f$ has integral coefficients, we have $f^{n_1p}(X)\equiv f^{n_1}(X^p) \mod p$.
So the congruence is implied by the equality
\[\left[f^{n_0}(X)f^{n_1}(X^p)\right]_0=\left[f^{n_0}(X)\right]_0\left[f^{n_1}(X)\right]_0,\]
which means: the product of a monomial from $f^{n_0}(X)$ and a monomial from $f^{n_1}(X^p)$ can never be constant, unless the two monomials are constant themselves.
It is this statement that we will prove now.\\
For the product of a non-constant monomial from $f^{n_0}(X)$ and a non-constant monomial from $f^{n_1}(X^p)$ to be constant, the monomial coming from $f^{n_0}(X)$ has to be a monomial in $X_1^p,...,X_n^p$, since all monomials in $f^{n_1}(X^p)$ are monomials in $X_1^p,...,X_n^p$.\\
A monomial 
\[M:=X^{\ell_1{\bf a}_1+\ell_2 {\bf a}_2+\ldots+\ell_m {\bf a}_m}=\prod_{j=1}^mX_1^{a_{1,j}\ell_j}...X_n^{a_{n,j}\ell_j}\] appearing in $f^{n_0}(X)$ corresponds to a partition 
\[n_0=\ell_1+...+\ell_m\] of $n_0$ in non-negative integers $\ell_i$.
If $M$ were a monomial in $X_1^p,...,X_n^p$, then we would have the divisibility
\[p \mid \sum_{j=1}^m a_{i,j}\ell_j \; \textnormal{for} \; 1 \leq i \leq n,\]
and hence
\[p \mid \gcd_{i=1,...,n}(\sum_{j=1}^ma_{i,j}\ell_j).\]
On the other hand, by \ref{reflexivsatz} we have
 \[\gcd_{i=1,...,n}(\sum_{j=1}^m a_{i,j}\ell_j)\leq \sum_{j=1}^m\ell_j=n_0<p.\]
So we conclude that $\sum_{i=1}^ma_{i,j}\ell_j=0$ for $1 \leq j \leq n$ and that the monomial 
$M$ is the constant monomial $X^0$.
Hence it follows that
\[\left[f^{n_0}(X)f^{n_1}(X^p)\right]_0=\left[f^{n_0}(X)\right]_0\left[f^{n_1}(X^p)\right]_0,\] and since
\[\left[f^{n_1}(X^p)\right]_0=\left[f^{n_1}(X)\right]_0,\]
the proposition follows.
\end{proof}

We remark that the congruence has the following interpretation.
By a result of \cite{DvK} (Theorem 4.) one can compactify the map $f:(\C^*)^n \lra \C$
given by the Laurent polynomial to a map $\phi:{\sX} \lra \P^1$ such that
the differential form $\Omega$ extends to a form in $\Omega^n(({\sX} \setminus \phi^{-1}(\{\infty\})))$. In the case $\Delta(f)$ is reflexive one has 
\[\deg(\pi_*\omega_{X/S})=1\]
see \cite{DK}, $(8.3)$.
On the other hand, from this and under an additional condition 
$(R)$, it follows from  \cite{yu} corollary 3.7 that the$\mod p$ 
Dwork-congruences hold.

\section{Strategy for higher $s$}
The idea for the higher congruences is basically the {\em same as for $s=1$}, but is 
combinatorically more involved. Surprisingly, one does not need any
statements stronger than \ref{reflexivsatz}.
To prove the congruence \ref{cong1}, we have to show that
\begin{equation} \left[\prod_{k=0}^sf^{n_k p^k}\right]_0\left[\prod_{k=1}^{s-1}f^{n_k p^{k-1}}\right]_0\equiv \left[\prod_{k=0}^{s-1}f^{n_kp^k}\right]_0\left[\prod_{k=1}^sf^{n_kp^{k-1}}\right]_0 \mod p^s.\label{cong1neu}\end{equation}
 To do this, we will use the following expansion of $f^{np^s}(X)$:
\begin{prop}We can write \begin{equation*}f^{np^s}(X)=\sum_{k=0}^sp^kg_{n,k}(X^{p^{s-k}}),\end{equation*}
where $g_{n,k}$ is a polynomial of degree $np^k$ in the monomials of $f$, independent of $s$, defined inductively by $g_{n,0}(X)=f^n(X)$ and
\begin{equation}\label{fentwicklung}p^kg_{n,k}(X):=f(X)^{np^k}-\sum_{j=0}^{k-1}p^jg_{n,j}(X^{p^{k-1-j}}).\end{equation}
\end{prop}
\begin{proof}
We have to prove that
the right-hand side of equation \ref{fentwicklung} is  divisible by $p^k$.
This is proved  by induction on $k$ and an application of  the congruence
\begin{equation}\label{copoly}f(X)^{p^m}\equiv f(X^p)^{p^{m-1}}\mod p^m. \end{equation}
For $k=1$, the divisibility follows directly by (\ref{copoly}). Assume that the statement is true for $m \leq k-1$.
Write 
$f(X)^{np^{k-1}}=\sum_{j=0}^{k-1}p^jg_{n,j}(X^{p^{k-1-j}}).$
Then,
$\sum_{j=0}^{k-1}p^jg_{n,j}(X^{p^{k-j}})=f(X^p)^{np^{k-1}}\equiv f(X)^{np^k}\mod p^n,$
and thus
$f(X)^{np^k}-\sum_{j=0}^{k-1}p^jg_{n,j}(X^{p^{k-j}})\equiv 0 \mod p^n$.\end{proof}

The congruences involve constant term expressions of the form
\begin{eqnarray}
\left[\prod_{k=a}^bf^{n_k p^k}\right]_0 & = & \left[\prod_{k=a}^b\sum_{j=0}^kp^jg_{n_k,j}(X^{p^{k-j}})\right]_0\nonumber\\
&=&\sum_{i_a\leq a}...\sum_{i_b \leq b}p^{\sum_{k=a}^bi_k}\left[\prod_{k=a}^bg_{n_k,i_k}(X^{p^{k-i_k}})\right]_0.\label{prodsumme}
\end{eqnarray}

Thus, equation (\ref{cong1neu}) translates into
\begin{small}
\begin{eqnarray}&&\label{congneu}
\sum_{i_0\leq 0}...\sum_{i_s \leq s}\sum_{j_1\leq 0}...\sum_{j_{s-1}\leq s-2}p^{\sum_{k=0}^si_k+\sum_{k=1}^{s-1}j_k}\left[\prod_{k=0}^sg_{n_k,i_k}(X^{p^{k-i_k}})\right]_0\left[\prod_{k=1}^{s-1}g_{n_k,j_k}(X^{p^{k-1-j_k}})\right]_0\nonumber\\&\equiv&\nonumber \\
&&\sum_{i_0\leq 0}...\sum_{i_{s-1} \leq s-1}\sum_{j_1\leq 0}...\sum_{j_{s}\leq s-1}p^{\sum_{k=0}^{s-1}i_k+\sum_{k=1}^{s}j_k}\left[\prod_{k=0}^{s-1}g_{n_k,i_k}(X^{p^{k-i_k}})\right]_0\left[\prod_{k=1}^{s}g_{n_k,j_k}(X^{p^{k-1-j_k}})\right]_0\nonumber\\
&&\mod p^s
\end{eqnarray}\end{small} 

Since this congruence is supposed to hold modulo $p^s$, on the left-hand side, 
only the summands with $\sum_{k=0}^si_k+\sum_{k=1}^{s-1}l_k\leq s-1$ contribute,
 and on the right-hand side, only those with $\sum_{k=0}^{s-1}i_k+\sum_{k=1}^sl_k \leq s-1$ play a role.\\
Now, we proceed by comparing these summands on both sides of equation 
\ref{cong1neu}. We will prove that each summand on the right-hand side 
is equal to exactly one summand on the left-hand side and vice versa.

\section{Splitting positions}

So we are led to study for $a \le b$ expressions of the type 
\[G(a,b;I) :=\left[\prod_{k=a}^{b}g_{n_k,i_k}(X^{p^{k-i_k}})\right]_0\]
where the $0 \le n_k \le p-1$ are fixed for $a \leq k \leq b$ and 
$I:=(i_a,...,i_b)$ is a sequence with $0 \le i_k \le k$. 

\begin{definition} We say that $G(a,b;I)$ {\em splits at $\ell$} if
\[G(a,b;I)=G(a,\ell-1;I)G(\ell,b;I)\] 
\end{definition}

The number of entries of $I$ is determined implicitly by $a$ and $b$, so that by $G(a,\ell-1;I)$ we mean the expression corresponding to the sequence $(i_a,...,i_{\ell-1})$, while by $G(\ell,b;I)$, we mean the expression corresponding to $(i_{\ell},...,i_b)$. Note that $\ell=a$ represents a trivial splitting, but splitting at $\ell=b$ is a non-trivial
property.

\begin{prop}\label{spalten}
If $k-i_k \ge \ell$ for all $k \ge \ell$, then $G(a,b;I)$ splits at $\ell$.
\end{prop}
\begin{proof}
A monomial $\prod_{j=1}^m(X^{p^{k-i_k}})^{{\bf a}_j \beta_{j,k}}$ occuring in
 $g_{n_k,i_k}(X^{p^{k-i_k}})$ corresponds to a partition
\[\beta_{1,k}+...+\beta_{m,k}=p^{i_k}n_k   \leq p^{i_k+1}-p^{i_k}\]
of the number $p^{i_k}n_k$ in non-negative integers $\beta_{1,k},...,\beta_{m,k}$. So we have
\[p^{k-i_k}(\beta_{1,k}+...+\beta_{m,k}\leq p^{k+1}-p^k.\]

It follows from the assumptions that the product
 $G(\ell,b;I)=\prod_{k=\ell}^bg_{n_k,i_k}(X^{p^{k-i_k}})$ is a Laurent-polynomial in $X^p$.  
As a consequence, the product of a monomial in $G(a,\ell-1;I)=\prod_{k=a}^{\ell-1}g_{n_k,i_k}(X^{p^{k-i_k}})$ and a monomial of $G(\ell,b;I)$ can be constant
only if the sum

\[
m_i:=\sum_{j=1}^mp^{a-i_a}a_{i,j}\beta_{j,a}+...+\sum_{j=1}^mp^{\ell-1-i_{\ell-1}}a_{i,j}\beta_{j,\ell-1}
\]
is divisible by $p^{\ell}$ for $1 \leq i \leq n$.

Set
\[ \gamma_j :=p^{a-i_a}\beta_{j,a}+...+p^{\ell-1-i_{\ell-1}}\beta_{j,\ell -1}   \]
so that
\[\sum_{j=1}^m a_{i,j}\gamma_j=m_i\]

 It follows that
\begin{small}
\[\sum_{j=1}^m\gamma_j=\sum_{j=1}^mp^{a-i_a}\beta_{j,a}+...+\sum_{j=1}^mp^{\ell -1-i_{\ell -1}}\beta_{j,\ell -1}\leq p^{a+1}-p^a+...+p^{\ell}-p^{\ell -1}=p^{\ell}-p^a<p^{\ell}.\] \end{small} 
Hence, it follows that
\[p^{\ell} \mid \gcd_{i=1,...,n}(\sum_{j=1}^ma_{i,j}\gamma_j)\leq \sum_{j=1}^m\gamma_j<p^{\ell},\] where the first inequality follows from
 Theorem \ref{reflexivsatz}. This implies $\sum_{j=1}^ma_{i,j} \gamma_j=0$ for $1 \leq i \leq n$. But this means that the monomial in  $\prod_{k=t}^{s-1}g_{n_k,i_k}(X^{p^{k-i_k}})$ is itself constant.
\end{proof}
Now that we know that we can split up expressions $G(a,b;I)$ satisfying the condition given in Proposition \ref{spalten}, we proceed by proving that all the summands on both sides of equation \ref{congneu} that do not have a coefficient divisible by $p^s$ satisfy this splitting condition.\\

\section{Three combinatorical Lemmas}

In this section, we prove three simple combinatorical lemmas which will be applied to split up expressions $G(0,s;I)G(1,s-1;J+1)$ that occur in the congruence (\ref{cong1neu}).

\begin{definition} Let $ a \le b$ and $I=(i_a,i_{a+1},\ldots,i_b)$ a sequence
with $0 \le i_k \le k$ for all $k$ with $a \le k \le b$.
We say that $\ell$ is a {\em splitting index for $I$} if $\ell >a$ and for $k \ge \ell$ one has
\[i_k \le k-\ell.\]
Remark that for a splitting index $\ell$ one can apply \ref{spalten} and that
$i_{\ell}=0$. 
\end{definition}

\begin{lemma}\label{spaltbedingung}
Let $I$ as above and assume that
\[\sum_{k=a}^b i_k \leq b-a-1.\]
Then there exists at least one splitting index for $I$.

%Then, there exists an index $k_0$ such that $i_{k_0}=0$ and $i_k\leq k-k_0$ for all $k>k_0$. 
%Then, $G(t,l;I)$ splits at least in one position.
\end{lemma}
\begin{proof}
Let $\sN:=\{k\;\;|i_k=0\}$ be the set of all indices $k$ such that the 
corresponding $i_k$ is zero. 
Since the sum has $b-a+1$ summands $i_k$, the set $\sN$ has at least two 
elements. So there exists at least one index $k \not=a$ such that $i_k=0$.\\
We will show by contradiction that one of these zero-indices is a splitting index.\\
We say that $\nu>k$ is a {\em violating index} with respect to $k \in \sN$ if $i_{\nu}>\nu-k $. Assume now that all $k \in \sN$ posses a violating index.\\
It follows directly that for each violating index $\nu$, $i_{\nu}\geq 2$. Furthermore, if $\nu$ is a violating index for $m$ different zero-indices $k_1<...<k_m$, it follows that $i_{\nu}\geq m+1$.\\
Now assume that we have $\mu$ different violating indices $\nu_1,...,\nu_{\mu}$ and that $\nu_j$ is a violating index for all $j\in \sN_j$, where we partition $\sN$ into disjoint subsets
\[\sN=\sN_1\cup \sN_2\cup...\cup \sN_{\mu}.\]
 Then $\sum_{j=1}^{\mu}i_{\nu_{j}}\geq \sum_{j=1}^\mu(\#\sN_j+1)=\#\sN+\mu$,
and

\[\sum_{k=a+1}^bi_k\geq \#\sN\cdot 0+\sum_{j=1}^{\mu}i_{\nu_j}+(b-a-(\#\sN+\mu))\cdot 1 =b-a>b-a-1,\]
a contradiction.\end{proof}

We can sharpen lemma \ref{spaltbedingung} to:
\begin{lemma}\label{spaltbedingung2}
Let $I$ be as above and assume that
\[\sum_{k=a}^b i_k=b-a-m.\]
Then there exist at least $m$ different splitting indices for $I$
%indices $k_j$ such that $i_{k_j}=0$ and $i_k\leq k-k_j$ for all $k>k_j$. 
%Then, the expression $G(t,l;I)$ splits in at least $m$ positions.
\end{lemma}
\begin{proof}
We proceed by induction on $m$. The case $m=1$ is just Lemma \ref{spaltbedingung}. Assume that for all $n \leq m$, we have proven the statement.\\
Now assume $\sum_{k=a}^bi_k=b-a-(m+1)$. Since $m+1>1$, there exists a splitting  index  $\nu$. We can split up the set of indices $\{i_a,...,i_b\}=\{i_a,...,i_{\nu-1}\}\cup\{i_{\nu},...,i_b\}$ in  position $\nu$ such that $\sum_{k=a}^{\nu-1} i_k=N_{\nu}$ and $\sum_{k=\nu}^bi_k=b-a-m-1-N_{\nu}$. Depending on $N_{\nu}$, we have to distinguish between the following cases:
\begin{enumerate}
\item $N_{\nu}>(\nu-1)-a-1$. It follows that $b-a-m-1-N_{\nu}<b-a-m-((\nu-1)-a-1)=b-m-(\nu-1)$, and thus $\sum_{k=\nu}^bi_k\leq b-\nu-m$. By induction, there exists at least $m$ splitting indices in $(i_{\nu},...,i_b)$, and thus for the whole  $(i_a,...,i_b)$, there exist at least $m+1$ such indices.
\item The case $N_{\nu}\leq (\nu-1)-a-1$ splits up in two cases:
\begin{enumerate}
\item $N_{\nu}\leq (\nu-1)-a-m$. By induction, $(i_a,...,i_{\nu-1})$ has  at least $m$ splitting indices, and the whole  $(i_a,...,i_b)$ has  at least $m+1$ such indices.
\item $N_{\nu}=(\nu-1)-a-n$, where $1\leq n\leq m$. Since $\sum_{k=a}^{\nu-1}i_k=(\nu-1)-a-n$, by induction for  $(i_a,...,i_{\nu-1})$ exist at least $n$ splitting indices. Since $\sum_{k=\nu}^bi_k=b-\nu-(m-n)$,  for  $(i_{\nu},...,i_b)$, there exist at least $m-n$ splitting indices. Thus, for the whole $(i_a,...,i_b)$ exist at least $n+(m-n)+1=m+1$ splitting indices.\end{enumerate} 
\end{enumerate}
\end{proof}

\begin{lemma}\label{gemeinsamsplit}
\begin{enumerate}
\item Let $I=(i_0,...,i_s)$ and $J=(j_1,...,j_{s-1})$ with  \[\sum_{k=0}^si_k+\sum_{k=1}^{s-1}j_k \leq s-1.\] Let $S_{I}$ be the set of splitting indices of $I$ and $S_J$ be the set of splitting indices of $J$. Then,
\[S_I\cap(S_J\cup \{1,s\})\not=\emptyset.\]

\item Let $I=\{i_0,...,i_{s-1}\}$ and $J=(j_1,...,j_{s})$ with  \[\sum_{k=0}^{s-1}i_k+\sum_{k=1}^{s}j_k \leq s-1.\] Let $S_{I}$ be the set of splitting indices of $I$ and $S_J$ be the set of splitting indices of $J$. Then,
\[(S_I\cup \{s\})\cap(S_J\cup \{1\})\not=\emptyset.\]

\end{enumerate}
\end{lemma}
\begin{proof}

\begin{enumerate}
\item Note that since $S_I\cup S_J\cup\{1,s\}\subset\{1,2,...,s\}$, it follows that $\#(S_I\cup S_J\cup\{1,s\})\leq s$. Note that $\sum_{k=0}^si_k\geq s-\#S_I$ by Lemma \ref{spaltbedingung2}. This implies that $\sum_{k=1}^{s-1}j_k\leq s-2-(s-(\#S_I+1))$, and hence that  $\#S_J\geq s-(\#S_I+1)$  by Lemma \ref{spaltbedingung2}.  But  $\#S_I+\#S_J+2=\#S_I+s-(\#S_I+1)+2=s+1>s$, which implies $\#(S_I\cap(S_J\cup\{1,s\}))\geq 1$, and thus the statement follows.
\item Note that since $(S_I\cup\{s\})\cup (S_J\cup\{1\})\subset\{1,...,s\}$, it follows that $\#(S_I\cup\{s\})\cup (S_J\cup\{1\})\leq s$. Now  $\sum_{k=0}^{s-1}i_k\geq s-1-\#S_I$, which implies $\sum_{k=1}^sj_k\leq s-1-(s-\#S_I-1)$, and  $\#S_J\geq s-\#S_I-1$. But $\#S_I+1+\#S_J+1\geq \#S_I+1+s-\#S_I=s+1>s$, which implies that $\#((S_I\cup\{s\})\cap(S_J\cup\{1\}))\geq 1$, and the statement follows.
\end{enumerate}
\end{proof}

\section{Proof for  higher $s$}
We will use the combinatorical lemmas on splitting indices from the last section to prove the congruence (\ref{cong1neu}) modulo $p^s$.\\
 For a sequence $I=(i_a,...,i_b)$, we write
\[p^I:=p^{\sum_{k=a}^bi_k}.\]
For a sequence $J=(j_a,...,j_b)$, we define $J+1:=(j_a+1,...,j_b+1)$. \\
Note that if $k-j_k>0$ for $a \leq k\leq b$, then we have
\begin{equation}G(a,b;J+1)=G(a,b;J),\label{gpraus}\end{equation} 
since the constant term of a Laurent-polynomial $f(X)$ is the same as the constant term of the Laurent-polynomial $f(X^p)$.\\
%Now, we define a map that maps tuple $I,J$ corresponding to a summand on the left-hand side of (\ref{congneu}) to a tuple $I',J'$ corresponding to a summand on the righthand side. We will prove that this map is a bijection.\\
Let  \[p^{I+J}G(0,s;I)G(1,s-1;J+1) \] be a summand on the left-hand side of (\ref{congneu}) defined by the tuple 
$(I,J)$ with $\sum_{k=0}^si_k+\sum_{k=1}^{s-1}j_k \leq s-1$, and let $1\leq \nu \leq s$  be such that $G(0,s;I)$ splits in position $\nu$ and either $G(1,s-1;J+1)$  splits in position $\nu$ or $\nu \in \{1,s\}$. Such a $\nu$ exists by Lemma (\ref{gemeinsamsplit}). Define $I'=(i_0',...,i_{s-1}')$ and $J'=(j_1',...,j_s')$ by 
\begin{eqnarray*}
i_k' & = & i_k \; \textnormal{for} \; k \leq \nu-1\\
i_k'& = & j_k \; \textnormal{for} \; k \geq \nu\\
j_k' & = & j_k \; \textnormal{for} \; k \leq \nu-1\\
j_k' & = & i_k \; \textnormal{for} \; k \geq \nu. 
\end{eqnarray*}
%Let \[A:=\{(I,J); i_k\leq k \; \textnormal{for}\; 0 \leq k \leq s, j_k\leq k-1 \; \textnormal{for}\; 1 \leq k \leq s-1\}\] and \[B:= \{(I,J); i_k\leq k \; \textnormal{for}\; 0 \leq k \leq s-1, j_k\leq k-1 \; \textnormal{for}\; 1 \leq k \leq s\}.\]\\
%Let $H:A\rightarrow B$ be the map given by $(I,J)\mapsto (I',J')$.
To show that $p^{I'+J'}G(0,s-1;I')G(1,s;J'+1)$ is in fact a summand on the right-hand side of (\ref{congneu}), we have to explain why $i_k'\leq k$ and $j_k'\leq k-1$. Note that $j_k\leq k-1$ for $1 \leq k \leq s-1$ and $i_k\leq k$ for $0 \leq k \leq s$. Furthermore, we have $i_k\leq k-1$ for $k\geq \nu$ since $i_{\nu}=0$ and $G(0,s;I)$ splits in position $\nu$, which means that $k-i_k\geq \nu\geq 1$ for $k\geq \nu$.\\
By definition of $j_k'$ and $i_k'$, it now follows that  $j_k'\leq k-1$ for $1\leq k \leq s$, and $i_k'\leq k$ for $0 \leq k \leq s-1$.\\
 Now that we know that $p^{I'+J'}G(0,s-1;I')G(1,s;J'+1)$ is in fact a summand on the right-hand side of congruence (\ref{congneu}), we prove the following Proposition.
Remark that obviously, we have $p^{I+J}=p^{I'+J'}$. 
\begin{prop}\label{fertig1} Let  $I,J,I'$ and $J'$ be defined as above. Then,
\[G(0,s,I)G(1,s-1;J+1)=G(0,s-1;I')G(1,s;J'+1).\] Thus, we can identify each summand on the left-hand side of (\ref{congneu}) with a summand on the right-hand side.
\end{prop}

\begin{proof}
By a direct computation:
\begin{eqnarray*}
&&G(0,s;I)G(1,s-1;J+1)\\ &=& G(0,\nu-1;I)G(\nu,s;I)G(1,\nu-1;J+1)G(\nu,s-1;J+1) \; \textnormal{by lemma} \; \ref{gemeinsamsplit}\\
&=& G(0,\nu-1;I)G(\nu,s;I+1)G(1,\nu-1;J+1)G(\nu,s-1;J) \,\textnormal{by}\;\; (\ref{gpraus})\\
&=& G(0,\nu-1;I)G(\nu,s-1;J)G(1,\nu-1;J+1)G(\nu,s;I+1)\; \textnormal{(commutation)}\\
&=&G(0,\nu-1;I')G(\nu,s-1;I')G(1,\nu-1;J'+1)G(\nu,s;J'+1)\, \textnormal{by definition of $I'$, $J'$}\\
&=& G(0,s-1;I')G(1,s;J'+1) \; \textnormal{by lemma} \; \ref{gemeinsamsplit},
\end{eqnarray*}
the statement follows.
Note that the last equality follows since by definition of $I'$ and $J'$, $i_{\nu}'=j_{\nu}'=0$, $k-i_k'\geq \nu$ and $k-j_k'\geq \nu$ for $k>\nu$. Thus, $G(0,s-1;I')$ and $G(1,s;J'+1)$ both split at $\nu$.\end{proof}

Since by Proposition \ref{fertig1}, we can identify every summand on the left-hand side of equation (\ref{congneu}) satisfying $I+J\leq s-1$ with a summand on the right-hand side, both sides are equal modulo $p^s$ and the proof of Theorem \ref{hauptsatz} is complete.\\

\noindent {\em Remark:} The above arguments to  prove the congruence $D3$ can be slightly simplified, as was shown to us by A. Mellit.

\section{An Example}
%\subsection{An Example}
Let $f$ be the Laurent-polynomial 
\begin{eqnarray*}f:&=&1/X_4+X_2+1/X_1X_4+1/X_1X_3X_4+1/X_1X_2X_3X_4+1/X_3+X_1/X_3\\&+&X_2/X_3X_4+X_1/X_3X_4+X_1X_2/X_3X_4+X_2/X_4
+1/X_2X_4+1/X_1X_2X_4+1/X_1X_2\\ &+&1/X_1+1/X_2X_3X_4+X_4+1/X_2 +X_1+X_1/X_4+1/X_3X_4+X_3+1/X_2X_3.\end{eqnarray*}

It is No. 24 in the list of Batyrev and Kreuzer \cite{BK}, so $\Delta(f)$
is a reflexive polytope and our theorem \ref{hauptsatz} applies:
the coefficients $a(n):=[f^n]_0$ 
\[ a(0)=1, a(1)=0, a(2)=18, a(3)=168, a(4)=2430, a(5)=37200, a(6)=605340\]
satisfy the congruence $D3$ modulo $p^s$ for arbitrary $s$.\\
The power series $\Phi(t)=\sum_{n=0}^{\infty}a(n)t^n$ is solution to a fourth order linear differential equation $ PF=0$, where the differential operator $ P$ is of Calabi-Yau type  
\begin{eqnarray*}
P&:=&88501054\theta^4+t(912382\theta(-291-1300\theta-2018\theta^2+1727\theta^3)+...\\ &+&t^{11}(3461674786667136(\theta+1)(\theta+2)(\theta+3)(\theta+4)),
\end{eqnarray*}
(where $\theta:=t\partial/\partial t$) that was determined in \cite{pm}.

\section{Behaviour under Covering} 
%The last example raises the question after a congruence among the $k-$fold coefficients if $a(n)\not=0$ implies $k|n$.
Let $f$ be a Laurent-polynomial corresponding to a reflexive polyhedron, let $\mathcal{A}$ be the exponent matrix corresponding to $f$,
and consider the vectors with integral entries in the kernel of $\mathcal{A}$. If there exists a positive integer $k$ such that
\[\ell:=\left(\begin{array}{c}\ell_1 \\ \vdots \\ \ell_m\end{array}\right) \in \ker(\mathcal{A}) \Rightarrow k| (\ell_1+...+\ell_m),\]
then it follows that 
\[a(n):=[f^n]_0\not=0\Rightarrow k|n,\]
since for $l \in \mathbb{N}$, 
\[[f^l]_0=\sum_{(\ell_1,...,\ell_m) \in A_{f,l}}{l \choose \ell_1,\ell_2,...,\ell_m},\]
where
\begin{eqnarray*}A_{f,l}&:=&\ker(\mathcal{A})\cap\{(\ell_1,...,\ell_m) \in \mathbb{N}_0^m, \ell_1+...+\ell_m=l.\}.
\end{eqnarray*}
We are interested in the congruences
\begin{eqnarray}\label{kkongr}\nonumber &&a(k(n_0+...+n_sp^s))a(k(n_1+...+n_{s-1}p^{s-2}))\equiv\\ && a(k(n_0+...+n_{s-1}p^{s-1}))a(k(n_1+...+n_sp^{s-1})) \mod p^s,
\end{eqnarray}
which we will prove in general for $s=1$, and  which we will prove for one example  by proving that the following condition is satisfied:
\begin{condition}\label{pggt} For a tuple $(\ell_1,...,\ell_m)$ with
\[\ell_1+...+\ell_m=k\mu\leq k(p-1),\]
it follows that 
\[p|\gcd(\sum_{j=1}^ma_{i,1}\ell_1,...,\sum_{j=1}^ma_{j,n}\ell_j)\Rightarrow \sum_{j=1}^ma_{i,1}\ell_j=...=\sum_{j=1}^ma_{j,n}\ell_j=0.\]
\end{condition}
Note that the proof is simliar for many other examples which we will not treat in here.\\
First of all, before we come to the example, we give a general proof of (\ref{kkongr}) for $s=1$.
\begin{prop} Let $a(n), n \in \mathbb{N}$ be an integral sequence satisfying
\[a(n_0+n_1p)\equiv a(n_0)a(n_1) \mod p\]
for $0 \leq n_0 \leq p-1$ and $a(n)\not=0$ iff $k|n$. Then

\[ a(k(n_0+n_1p))\equiv a(kn_0)a(kn_1) \mod p.\] 
\end{prop}
\begin{proof}
If $kn_0<p$, then the proposition follows directly. Hence assume that $kn_0=n_0'+n_0''p>p-1$. Then
\begin{eqnarray*}a(k(n_0+n_1p))&=&a(n_0'+(kn_1+n_0'')p)\\ &\equiv& a(n_0')a(kn_1+n_0'') \mod p.\end{eqnarray*}
Since $k\not|n_0'$ and $a(n_0')=0$ by assumption,  it follows that
\[a(k(n_0+n_1p))\equiv 0 \mod p.\]
On the other hand, $a(kn_0)=a(n_0'+n_0''p)\equiv a(n_0')a(n_0'') \mod p$ where $a(n_0')=0$, and thus $a(kn_0)\equiv 0 \mod p$ and
\[a(kn_0)a(kn_1)\equiv 0 \mod p\] and the proposition follows.
\end{proof}

\subsection{An Example}

Let $f$ be the Laurent-polynomial No. 62 in the list of Batyrev and Kreuzer \cite{BK}, which is given by
\[f:=X_1+X_2+X_3+X_4+\frac{1}{X_1X_2}+\frac{1}{X_1X_3}+\frac{1}{X_1X_4}+\frac{1}{X_1^2X_2X_3X_4}.\]
Then, the coefficients $a(n)$ are given by $a(n)=0$ if $n \not=0 \mod 3$ and 
\[a(3n)=\frac{(3n)!}{n!^3}\sum_{k=0}^n{n \choose k}^2{n+k \choose k}.\]
The Newton polyhedron $\Delta(f)$ is reflexive (see \cite{BK}), and hence by Theorem \ref{hauptsatz}, the coefficients $a(n)$ satisfy the congruence $D3$ modulo $p^s$ for arbitrary $s$.\\
The power series $\Phi(t)=\sum_{n=0}^{\infty}a(3n)t^n$ is solution  to a fourth order linear differential equation $ PF=0$, where the differential operator $ P$ is of Calabi-Yau type  and is given by 
\begin{eqnarray*}P&:=&\theta^4-3t(3\theta+2)(3\theta+1)(11\theta^2+11\theta+3)\\ &-&9t^2(3\theta+5)(3\theta+2)(3\theta+4)(3\theta+1).\end{eqnarray*}
In this example,
 the exponent matrix is
\[\mathcal{A}:=\left(\begin{array}{cccccccc}1&0&0&0&-1&-1&-1&-2\\ 0 &1&0&0&-1&0&0&-1\\0&0&1&0&0&-1&0&-1\\0&0&0&1&0&0&-1&-1\end{array}\right).\]
A basis of $\ker(\mathcal{A})$ is given by
\[\{\left(\begin{array}{c}1\\1\\0\\0\\1\\0\\0\\0\end{array}\right),\left(\begin{array}{c}1\\0\\1\\0\\0\\1\\0\\0\end{array}\right),\left(\begin{array}{c}1\\0\\0\\1\\0\\0\\1\\0\end{array}\right),\left(\begin{array}{c}2\\1\\1\\1\\0\\0\\0\\1\end{array}\right)\},\]
and thus it follows that $[f^n]_0\not=0\Rightarrow 3|n$ and $k=3$.  We prove that Condition \ref{pggt} is satisfied in this example. 
Assume that $p\not=3$ and that \[p|\gcd(\sum_{j=1}^8a_{1,j}\ell_j,...,\sum_{j=1}^8a_{4,j}\ell_j)\; \textnormal{for} \; \ell_1+...+\ell_8=3\mu\leq 3(p-1).\]
This means that there exist $x_1,x_2,x_3,x_4 \in \mathbb{Z}$ such that
\begin{eqnarray*}
\ell_1&=&\ell_5+\ell_6+\ell_7+2\ell_8+x_1p\\
\ell_2&=&\ell_5+\ell_8+x_2p\\
\ell_3&=&\ell_6+\ell_8+x_3p\\
\ell_4&=&\ell_7+\ell_8+x_4p,
\end{eqnarray*}
which implies
\begin{equation*}3(\ell_5+\ell_6+\ell_7+2\ell_8)+(x_1+x_2+x_3+x_4)p=3\mu \leq 3(p-1).\end{equation*}
Thus, it follows that $(x_1+...+x_4)=3z$ for some $z \in \mathbb{Z}$ and that
\begin{equation*}\ell_5+\ell_6+\ell_7+2\ell_8+zp=\mu \leq p-1.\end{equation*} 
Since $\ell_5,...,\ell_8$ are nonnegative integers, it follows directly that $z\leq 0$. Now, consider the following cases:
\begin{enumerate}
\item $z=0$: Then, \begin{equation}\label{klgl1}\ell_5+\ell_6+\ell_7+2\ell_8\leq p-1\end{equation} Assume that $x_i<0$, i.e. $x_i \leq -1$ for some $1 \leq i \leq 4$. Since $\ell_1,...,\ell_4$ are nonnegative integers, it follows that either $\ell_5+\ell_6+\ell_7+2\ell_8\geq p$ or $\ell_j+\ell_8\geq p$ for some $5 \leq j \leq 7$, a contradiction to (\ref{klgl1}). Thus, since $x_1+x_2+x_3+x_4=0$, it follows that $x_1=x_2=x_3=x_4=0$ and that
\[\sum_{j=1}^8a_{1,j}\ell_j=...=\sum_{j=1}^8a_{4,j}\ell_j=0\]
in this example.
\item $z<0$: Assume that $\ell_5+\ell_6+\ell_7+2\ell_8<(-z+1)p$. Since $\ell_1\geq 0$, it follows that $x_1>z-1$, and since $x_1$ is integral, that $x_1\geq z$. Since $x_1+x_2+x_3+x_4=3z$, it follows that $x_2+x_3+x_4 \leq 2z$. Now assume that $x_i\geq z$ for $2 \leq i \leq 4$. Then $x_2+x_3+x_4\geq 3z$,a contradiction. Hence there exists an index $i$ such that $x_i<z$, and hence $x_i\leq z-1$. Since $\ell_i\geq 0$, it follows that $\ell_{i+2}+\ell_8\geq (-z+1)p$, a contradiction since $\ell_{i+2}+\ell_8\leq \ell_5+\ell_6+\ell_7+2\ell_8<(-z+1)p$ by assumption. Thus, we have $\ell_5+\ell_6+\ell_7+2\ell_8\geq(-z+1)p$, which implies
$p\leq \ell_5+\ell_6+\ell_7+2\ell_8+zp\leq p-1$, a contradiction.
\end{enumerate}
Thus, it follows that the only possible case is $z=0$, and $x_1=x_2=x_3=x_4=0$, which proves that Condition \ref{pggt} is satisfied in this example.
\section{The statement D1}
For the proof of congruence (\ref{cong1}), the coefficients $c_{\bf a}$ of
\[f(X)=\sum_{\bf a}c_{\bf a}X^{\bf a}\]
did not play a role. This is different if one is interested in the proof of part D1 of the Dwork congruences. Let $n \in \mathbb{N}$, and write $n=n_0+pn_1$, where $n_0\leq p-1$. Then, to prove D1 for the sequence $a(n):=[f^n]_0$ means that one has to prove that
\begin{equation}\label{D1neu}\frac{\left[f^{n_0+n_1p}\right]_0}{\left[f^{n_1}\right]_0}\in \mathbb{Z}_p.\end{equation} 
Sticking to the notation of the previous sections, we write
\begin{equation}\label{D1crit}f^{n_0+n_1p}(X)=f^{n_0}(X)f^{n_1}(X^p)+pf^{n_0}(X)g_{n-1,1}(X).\end{equation}
Assume that $p^k|[f^{n_1}]_0$. To prove (\ref{D1neu}), one has to prove that
$p^k|[f^{n_0+n_1p}]_0$. By (\ref{D1crit}), this is equivalent to proving that $p^{k-1}|[f^{n_0}g_{n_1,1}(X)]_0$. Thus, the proof of part D1 of the Dwork congruences requires an investigation in the $p-$adic orders of the constant terms of $f^{n_1}$ and $g_{n_1,1}$ for arbitrary $n_1$, and requires methods that are completely different from the methods we applied to prove the congruence $D3$.\\

\noindent {\bf Acknowledgement.} We thank A. Mellit for his comments. The work of the first author was funded by the SFB Transregio 45.


\begin{thebibliography}{99}
\bibitem[AESZ]{AESZ}{\sc G. Almkvist, C. van Enckevort, D. van Straten, W. Zudilin}, {\em Tables of Calabi-Yau Equations},
{\tt arXiv:math/0507430}.
\bibitem[BK]{BK} {\sc Batyrev, Kreuzer}, {\em Constructing new Calabi-Yau 3-folds and their mirrors via conifold transitions }
{\tt arXiv:0802.3376v2}
\bibitem[DK]{DK} C. Doran, M. Kerr, {\em Algebraic K-theory of toric hypersurfaces}, {\tt ArXiv 0809.4669v1}.
\bibitem[DvK]{DvK} J. Duistermaat, W. van der Kallen, {\em Constant terms in powers of a Laurent polynomial}, Indagationes Math. {\bf 9}, (1998), 221-231.
\bibitem[Dw]{dwork} {\sc B. Dwork}, {\em p-adic cycles}, Publications math\'ematiques de l'I.H.E.S., tome 37, 1969, 27-115.
\bibitem[PM]{pm} P. Metelytsin, in preparation.
\bibitem[SvS]{SvS} K. Samol, D. van Straten, {\em Frobenius Polynomials for Calabi-Yau equations}, Communications in Number Theory and Physics, vol 2, no 3, (2008)
\bibitem[Yu]{yu} {\sc Jeng-Daw Yu}, {\em Notes on Calabi-Yau ordinary differential equations} {\tt arXiv:0810.4040v1}
\end{thebibliography}
\end{document}